\documentclass{article}
\usepackage[utf8]{inputenc}
\usepackage{amsmath, amssymb, enumerate, multicol, chngcntr}
\usepackage[exercise]{amsthm}
\usepackage{xcolor}
\usepackage{pdfpages}
\usepackage{comment}
\usepackage{caption}
\usepackage{bm}
\usepackage{mathtools}

\theoremstyle{plain}
\newtheorem{thm}{Theorem}

\newtheorem{lemma}{Lemma}

\counterwithin{equation}{section}

\usepackage{fancyhdr}
\usepackage{setspace}

\usepackage[english]{babel}
\usepackage[euler]{textgreek}
\usepackage{stmaryrd}

\usepackage{tikz}
\usetikzlibrary{shapes,arrows}
\usepackage{tikz-cd}
\usepackage{amssymb}
\usepackage{graphicx}
\graphicspath{{Images/}}
\usepackage[font=small,labelfont=bf]{caption}
\usepackage{pgf,tikz,pgfplots}
\pgfplotsset{compat=1.15}
\usepackage{mathrsfs}
\usepackage{diagbox}
\usepackage{float}
\usepackage{xcolor}
\usepackage{hyperref}
\usepackage{ragged2e}
\usepackage{enumerate}
\usepackage{upgreek}
\usepackage{multicol}

\theoremstyle{remark}

\theoremstyle{plain}

\newtheoremstyle{note}
  {3pt}
  {3pt}
  {}
  {}
  {\itshape}
  {:}
  {.5em}
  {}

\newtheoremstyle{citing}
  {3pt}
  {3pt}
  {\itshape}
  {}
  {\bfseries}
  {.}
  {.5em}
  {\thmnote{#3}}

\theoremstyle{citing}

\newtheoremstyle{break}
  {9pt}
  {9pt}
  {\itshape}
  {}
  {\bfseries}
  {.}
  {\newline}
  {}

\let\lvert=|\let\rvert=|

\title{On the Symmetric Normaliser Graph of a Group}

\author{Surbhi \footnote{Corresponding author, Dr. B. R. Ambedkar University Delhi, Delhi 110006; \ E-mails: surbhi.21@stu.aud.ac.in, surbhi.ts19@gmail.com} \ and \ Geetha Venkataraman\footnote{Dr. B. R. Ambedkar University Delhi, Delhi 110006; E-mails: geetha@aud.ac.in, geevenkat@gmail.com}}

\date{}
\begin{document}
\fontfamily{cmr}\selectfont
\maketitle

\bigskip
\noindent
{\small{\bf ABSTRACT:}}
 In this paper we introduce the symmetric normaliser graph of a group $G$. The vertex set of this graph consists of elements of the group. Vertices $x$ and $y$ are adjacent if $x$ lies in the normaliser of $\langle y \rangle$ and $y$ lies in the normaliser of $\langle x \rangle$. We investigate the hierarchical position this graph occupies in the hierarchy of graphs defined on groups. We show that the existing hierarchy is further refined by this graph and that the edges of this graph lie between the edges of the commuting graph and the nilpotent graph. For finite groups, we prove a necessary and sufficient condition for the symmetric normaliser graph to be equal to the commuting graph and similarly, for equality with the nilpotent graph. The edge set of the symmetric normaliser graph is also a subset of the edge set of the Engel graph of a group and has connections to the non-generating graph of a group.


\medskip
\noindent
{\small{\bf Mathematics Subject Classification-MSC2020}{:} }
05C25, 20D60, 20E34, 20F18

\baselineskip=\normalbaselineskip


\section{Introduction}
The study of graphs defined on groups has emerged as an important area of research. There is a rich literature exploring graphs defined on groups. Some of these are \cite {AA2007},\cite {AACNS2017}, \cite{PJC2022}, \cite{ADM2023} and \cite{DGLM2025}. In particular, in the survey paper \cite{PJC2022}, P. J. Cameron introduced the concept of a hierarchy between graphs defined on groups and posed several related open questions. Some of these open questions have been answered in \cite{SV2024_1} and \cite{SV2024_2}.

We define a new graph called the \textit{symmetric normaliser graph} of a finite group $G$. It is a graph with vertex set $G$ and two vertices $x$ and $y$ are adjacent if $x$ lies in $N(\langle y \rangle)$ and $y$ lies in $N(\langle x \rangle)$, where $N(H)$ denotes the normaliser of the subgroup $H$ of $G$. We denote this graph as SNorm($G$). The paper \cite{DGM2025} independently investigates a directed graph, which has some similarity, called directed normalising graph. 
\par
 We will denote the power graph, the enhanced power graph and the commuting graph of a group $G$ by Pow($G$), EPow($G$) and Com($G$) respectively. In \cite{PJC2022}, P. J. Cameron described the sequence: power graph, enhanced power graph and commuting graph as a hierarchy, as the edge set of each is contained in the edge set of the next in the sequence. 

In addition to this, we consider the nilpotent graph denoted as Nilp($G$) (in which two vertices are adjacent if they generate a nilpotent subgroup).  The hierarchy we can consider now is the sequence: power graph, enhanced power graph, commuting graph and nilpotent graph. We refine this for a group $G$ with the introduction of the symmetric normaliser graph. Before we state our result refining the above hierarchies we mention the Engel graph of a group denoted as Engel($G$) and the non-generating graph of a group $G$, denoted by NGen($G$). 

The Engel graph was first introduced in \cite{AA2007} and was modified in \cite{PJC2022}. We follow the second definition. Engel($G$) has the elements of $G$ as its vertex set and elements $x$ and $y$ are adjacent if for some integer $k$, either $[x ,_k y] = 1$ or $[y ,_k x] = 1$ where $[x ,_k y] = [[x ,_{k-1} y],y]$. The graph NGen($G$) has $G$ as the vertex set and two elements are adjacent if they do not generate $G$. For a graph $\Gamma$ let $E(\Gamma)$ denote its edge set.

\begin{thm}\label{thm_1}
    Let $G$ be a group. Then
    \begin{enumerate}[{\rm (i)}]
        \item $E$\rm{(Com(}$G$\rm{))} $\subseteq$ $E$\rm{(SNorm(}$G$\rm{))} $\subseteq$ $E$\rm{(Nilp(}$G$\rm{))}.
        \item  $E$\rm{(SNorm(}$G$\rm{))} $\subseteq$ $E$\rm{(NGen(}$G$\rm{))} if $G$ is a non-abelian simple group or not 2-generated. 
        \item $E$\rm{(SNorm(}$G$\rm{))} $\subseteq$ $E$\rm{(Engel(}$G$\rm{))}.
    \end{enumerate}
\end{thm}

The natural question that arises is when do we have equality amongst the edge sets of graphs in the above hierarchy. For a finite group $G$ to have equal power and enhanced power graph, or equal enhanced power graph and commuting graph, the necessary and sufficient conditions can be found in \cite{PJC2022}. These were extended to groups in general in \cite{SV2024_2}. We have the following results in a similar vein. We need to discuss exponent-critical groups \cite{BCMV2025} and the definition of an SNNC-group (symmetric normaliser non-commuting group) to state our second result.

A finite group $G$ is exponent-critical if the exponent of $G$ is not the least common multiple of the exponents of its proper non-abelian subgroups. An exponent-critical $p$-group $P$ is of type $\mathcal{B}$ if $P$ has more than one abelian maximal subgroup. An SNNC-group is an exponent-critical $p$-group of type $\mathcal{B}$, which is isomorphic to the quaternion group $Q_8$ or to a group $P$ of order $p^n$ given below.
$$P =\langle a,b \mid b^{p^{\beta}} = [a,b]^p = [a,b,a] = [a,b,b] = 1, a^{p^{\alpha}} = [a,b]  \rangle$$ where $\alpha, \beta$ are integers such that: $\alpha + \beta = n-1$. Further $\alpha \geq \beta \geq 1$ except when $p=2$ and $\alpha = \beta$ we have $\alpha > 1$.

\begin{thm}\label{thm_2}
    The commuting graph of a finite group $G$ is equal to the symmetric normaliser graph of $G$ if and only if $G$ does not have a subgroup isomorphic to an SNNC-group.
\end{thm}

Lastly, we have the following necessary and sufficient condition.

\begin{thm}\label{thm_3}
    The symmetric normaliser graph and the nilpotent graph of a finite group $G$ are equal if and only if for any odd prime $p$, the Sylow $p$-subgroups are abelian and for $p=2$ the Sylow $p$-subgroups are either abelian or a direct product of $Q_8$, the quaternion group and an elementary abelian group.
\end{thm}

In Section 2 we prove Theorem \ref{thm_1} and Theorem \ref{thm_2}. The last section has the proof of Theorem \ref{thm_3}, some concluding remarks and open questions.  

\section{Refined Hierarchies and Equalities I}

\setcounter{thm}{0}
\par
Proposition 2.6 of \cite{PJC2022} states that if a finite group $G$ is non-abelian or not 2-generated, then $E$(Com($G$)) $\subseteq$ $E$(NGen($G$)). We derive a similar result for the SNorm($G$) and NGen($G$).
\begin{thm}\label{Hierarchy} 
    Let $G$ be a group. Then
    \begin{enumerate}[{\rm (i)}]
        \item $E$\rm{(Com(}$G$\rm{))} $\subseteq$ $E$\rm{(SNorm(}$G$\rm{))} $\subseteq$ $E$\rm{(Nilp(}$G$\rm{))}.
        \item  $E$\rm{(SNorm(}$G$\rm{))} $\subseteq$ $E$\rm{(NGen(}$G$\rm{))} if $G$ is a non-abelian simple group or not 2-generated. 
        \item $E$\rm{(SNorm(}$G$\rm{))} $\subseteq$ $E$\rm{(Engel(}$G$\rm{))}.
    \end{enumerate}
\end{thm}
\begin{proof}
It is clear that if an edge is present in the commuting graph then it is present in the symmetric normaliser graph. Let $x$ and $y$ be adjacent in SNorm($G$) and let $H$ denote the subgroup of $G$ generated by $x$ and $y$. Since the commutator $c = [x,y] = xyx^{-1}y^{-1}$ commutes with both $x$ and $y$, the subgroup generated by $c$ is in the center of $H$. Clearly, the quotient group $\frac{H}{\langle c \rangle}$ is abelian and therefore nilpotent. Consequently $H$ is nilpotent.

Now let $G$ be a non-abelian simple group and assume that $x,y$ are non-identity elements.  Then we have $H \leq N(\langle y \rangle) \lneq G$, which implies that $x,y$ are adjacent in NGen($G$). Lastly, since the commutator $[x, y]$ commutes with both $x$ and $y$, we have $[x,_2 y ] = 1$ and $[y ,_2 x] =1$. Therefore, elements $x,y$ are adjacent in the Engel graph of $G$.
\end{proof}

The rest of this section is devoted to proving results which will give us the necessary and sufficient condition under which the commuting graph of a finite group equals that of the symmetric normaliser graph. The commutator identities below will also be used.
\begin{subequations}
\begin{equation}
\label{comu_id_1}
[x,yz] =[x,y][x,z]^y
\end{equation}
\begin{equation}
\label{comu_id_2}
 [xz,y] = [z,y]^x[x,y]
 \end{equation}
\end{subequations}

\begin{lemma}\label{com_eq_snorm_lemma}
    Let $G$ be a finite group such that the commuting graph of $G$ is not equal to the symmetric normaliser graph of $G$. Then, the group $G$ has a subgroup $H = \langle h, k \rangle$ which is a 2-generated $p$-group such that $|H'| = p$ and $H' \leq Z(H)$. Further, $h, k$ are adjacent in SNorm($H$).
\end{lemma}
\begin{proof}
    Let $G$ be a finite group such that the edge set of its commuting graph is a proper subset of the edge set of its symmetric normaliser graph. This gives us elements $x$ and $y$ that satisfy the \ref{snorm_property} given below. 
    \begin{equation}\label{snorm_property}
        ab \neq ba \mbox{ , } a \in N(\langle b \rangle) \mbox{ and }  b \in N(\langle a \rangle).
    \end{equation}
     Note that if $x, y$ satisfy \ref{snorm_property} then $[x, y]$ commutes with both $x$ and $y$. So using (2.1a) and (2.1b) we get that $ [x^m, y]= {[x, y]}^m = [x, y^m]$ for any integer $m$. 
     
     Let $\mathcal{L}$ denote the collection of all elements $x'$ such that there exists $y' \in G$ with $x'$,$y'$ satisfying \ref{snorm_property}. Choose $h$ from $\mathcal{L}$ such that $|h|$ is the least. Now, define a set $\mathcal{L}_h$ as a collection of all elements $y''$ such that $h$ and $y''$ satisfy \ref{snorm_property}. From this set, we choose $k$ such that $|k|$ is the least. Let $p$ be a prime that divides the order of $h$. Now $h^p$ lies in normaliser of $\langle k \rangle$ in $G$ and $k$ lies in normaliser of $\langle h^p \rangle$ in $G$. This implies that $h^p$ commutes with $k$ and we get $[h,k]^p=e$. Similarly, if $q$ is a prime that divides the order of $k$, then $[h,k]^q=e$ which means that $p$ and $q$ are equal. This gives that the orders of $h$ and $k$ are of prime-power for the same prime, say $p$. Let $H=\langle h,k \rangle$. Then the subgroup $H$ has order $p^n$ for some natural number $n$ and its derived subgroup $H' = \langle [h,k] \rangle$ has order $p$ and $H' \leq Z(H)$, which makes $H$ a $2$-generated nilpotent group of class 2. \\
\end{proof}

Simon R. Blackburn et. al. \cite{BCMV2025} defined a new class of finite groups called exponent-critical $p$-groups of type $\mathcal{B}$. By Theorem D of \cite{BCMV2025}, a non-abelian finite $p$-group has type $\mathcal{B}$ if and only if it is 2-generated with derived subgroup of order $p$. This implies that if $G$ is a finite group such that its commuting graph and symmetric normaliser graph are not equal, then $G$ has a subgroup which is an exponent-critical $p$-group of type $\mathcal{B}$. 
\par
A parameter-based list of isomorphism classes of exponent-critical $p$-groups of type $\mathcal{B}$ was given as Theorem 4.2 of \cite{BCMV2025}, which we reproduce below.
\begin{lemma}\label{ECG_thm}
    Let $P$ be an exponent-critical finite non-abelian $p$-group of type $\mathcal{B}$ of order $p^n$. Then 
    $$P \cong \langle a,b \mid [a,b]^p = [a,b,a] = [a,b,b] = 1, a^{p^{\alpha}} = [a,b]^{p^{\rho}}, b^{p^{\beta}} = [a,b]^{p^{\sigma}} \rangle$$ where $\alpha, \beta, \rho, \sigma$ are integers such that: $ \alpha \geq \beta \geq 1, \alpha + \beta = n-1$ and $ 0 \leq \rho, \sigma \leq 1 $. When $p$ is odd, $P$ is isomorphic to exactly one of the groups whose parameters $(\alpha, \beta, \rho, \sigma)$ are listed below:
    \begin{enumerate}
        \item[A1.]   \begin{enumerate}
                    \item $(\alpha, \beta, 0, 1)$ with $\alpha > \beta \geq 1$. \label{p_1a}
                    \item $(\alpha, \beta, 1, 1)$ with $\alpha > \beta \geq 1$. \label{p_1b}
                    \item $(\alpha, \beta, 1, 0)$ with $\alpha > \beta \geq 1$. \label{p_1c}
                \end{enumerate}
        \item[A2.]  \begin{enumerate}
                    \item $(\alpha, \alpha, 0, 1)$ with $\alpha \geq 1$. \label{p_2a}
                    \item $(\alpha, \alpha, 1, 1)$ with $\alpha \geq 1$. \label{p_2b}
                \end{enumerate}
    \end{enumerate}
    When $p=2$, $P$ is isomorphic to exactly one of the groups whose parameters $(\alpha, \beta, \rho, \sigma)$ are listed below:
    \begin{enumerate}
        \item[B1.]   \begin{enumerate}
                    \item $(\alpha, \beta, 0, 1)$ with $\alpha > \beta \geq 1$. \label{2_1a}
                    \item $(\alpha, \beta, 1, 1)$ with $\alpha > \beta \geq 1$. \label{2_1b}
                    \item $(\alpha, \beta, 1, 0)$ with $\alpha > \beta \geq 1$. \label{2_1c}
                \end{enumerate}
        \item[B2.]   \begin{enumerate}
                    \item $(\alpha, \alpha, 0, 1)$ with $\alpha > 1$. \label{2_2a}
                    \item $(\alpha, \alpha, 1, 1)$ with $\alpha > 1$. \label{2_2b}
                \end{enumerate}
        \item[B3.]   \begin{enumerate}
                    \item $(1,1,0,0)$. \label{2_3a}
                    \item $(1,1,1,1)$. \label{2_3b}
                \end{enumerate}
    \end{enumerate}
\end{lemma}

The authors show that each element of $P$ can be written uniquely as $a^ib^j [a, b]^k$ where $ 0 \leq i < p^{\alpha}, 0 \leq j < p^{\beta}, 0 \leq k <p$. It is also shown that $P' = \langle [a, b] \rangle \leq Z(P)$. This makes $P$ nilpotent of class 2. Also exp($P$), the exponent of $P$, is equal to max($|a|$,$|b|$) except when  $p=2$ and $(\alpha, \beta, \rho, \sigma)=(1,1,1,1)$ and then exp($P$) = $2^2=4$. 

From the list in Lemma \ref{ECG_thm}, we identify the groups that do not have elements $x, y$ which generate $P$ and are also adjacent in the SNorm($P$). This is the main thrust of our next result. Note that in a nilpotent group of class at most 2, 
\begin{equation}\label{id_ab}
    (ab)^i=[b,a]^{i(i-1)/2}a^ib^i
\end{equation} for any positive integer $i$.

\begin{lemma}\label{ecg_case_1b_1c_2b}
    If $P$ is an exponent-critical group of type $\mathcal{B}$ of order $p^n$ with parameters as in Lemma \ref{ECG_thm}, cases A1(b), B1(b); A1(c), B1(c); A2(b), B2(b); B3(b), then there do not exist elements $x,y \in P$ such that $P= \langle x,y \rangle$ and $x,y$ are adjacent in SNorm($P$).
\end{lemma}

\begin{proof}
Consider the parameters in cases A1(b) and B1(b). Then $\sigma = 1$ and $\rho = 1$. Let $x=a^ib^j[a,b]^k$ and $y=a^rb^s[a,b]^t$ be two distinct elements of $P$ such that they are adjacent in SNorm($P$). We shall prove that they commute with each other, so they can not generate the non-abelian group $P$. It follows by induction on $m$ that
    \begin{equation}\label{1b_ym_ind}
        y^m = a^{rm}b^{sm}[a,b]^{tm-rsm(m-1)/2}.
    \end{equation}
The elements $x$ and $y$ are adjacent in SNorm($P$), so there must exist integers $d$ and $e$ such that $$yxy^{-1}=x^d \mbox{ and } xyx^{-1} =y^e.$$ 
    By substituting values of $x$,$y$ and using $b^lab^{-l} = a[b,a]^l$, we get 
    \begin{equation}\label{1b_xyx_inv}
        xyx^{-1} = a^rb^s[b,a]^{rj-si-t}.
    \end{equation}
    Equating \ref{1b_xyx_inv} and \ref{1b_ym_ind} for $m=e$, we get $re \equiv r \mbox{ mod } p^{\alpha}$. If $p$ does not divide $r$, we get $y^e=y$ and so $xyx^{-1}=y$. Similarly, if $p$ does not divide $i$, we have $x^d=x$, and so $yxy^{-1}=x$. Assume that $p$ divides both $r$ and $i$, then $p$ divides $rj-si$ and $[x,y] = [b,a]^{rj-si} = 1$. Hence, $x$ and $y$ generate an abelian subgroup of $P$ whenever they are adjacent in SNorm($P$).
    
    For cases A1(c) and B1(c) we have $\rho = 1$ and $\sigma = 0$. Let $x,y \in P$. Then, there exist unique $i,j,r,s$ satisfying $1 \leq i,r \leq p^{\alpha}$ and $1 \leq j,s \leq p^{\beta + 1}$ such that $x=a^ib^j$ and $y=a^rb^s$. Let $x,y$ be adjacent in SNorm($P$). We shall prove that they commute with each other. Since $x$ and $y$ are adjacent in SNorm($P$), then there must exist integers $d$ and $e$ such that $$yxy^{-1}=x^d \mbox{ and } xyx^{-1} =y^e.$$ 
Proceeding as in the previous case we get $xy=yx$.

Lastly we consider cases A2(c), B2(c) and B3(b). So we have $(\alpha, \beta, \rho, \sigma) = (\alpha, \alpha, 1, 1)$ with $\alpha \geq 1$. If $x$ and $y$ are adjacent in SNorm($P$), then $|\langle x,y \rangle| = \frac{|x||y|}{|\langle x \rangle \cap \langle y \rangle|}  \leq p^{2\alpha} = p^{n-1} < |P|$.  
\end{proof}
This leaves us with the cases when $\rho=0, \sigma=1$ for any prime $p$ and $(\alpha, \beta, \rho, \sigma) = (1, 1, 0, 0)$ when $p=2$. These are the cases A1(a), A2(a), B1(a), B2(a) and B3(a) of Lemma \ref{ECG_thm}. The parameters in these cases are precisely those used to define the SNNC-groups in the introduction.
 
\begin{lemma}\label{ECG_snorm}
    Let $P$ be an exponent-critical non-abelian $p$-group of type $\mathcal{B}$ of order $p^n$. Then the group $P$ has elements $x,y$ which generate $P$ and are adjacent in SNorm($P$) if and only if $P$ is a SNNC-group.
\end{lemma}
\begin{proof}
We will prove this result by showing the existence of $x$ and $y$ which satisfy the conditions in each of the SNNC-groups.
    
 Note that in the case B3(a), when $(\alpha, \beta, \rho, \sigma) = (1, 1, 0, 0)$, we have that $P$ is isomorphic to the quaternion group $Q_8$. The symmetric normaliser graph of $Q_8$ is a complete graph. So $Q_8$ has 2 elements that generate it and are also adjacent in SNorm($Q_8$).
    
For the parameters that don't give $Q_8$, the proof is as follows. We have $ [a,b] = a^{p^\alpha}$ which gives us $ba^{-1}b^{-1}=a^{p^{\alpha}-1}$. Thus, 
    \begin{equation}\label{p_1a_eq1}
      a(ab)a^{-1} = a^2(ba^{-1}b^{-1})b = a^2(a^{-1+p^{\alpha}})b=a^{1+p^{\alpha}}b.
    \end{equation}
Since $\langle a \rangle$ is normal in $P$, we have $ab \in N(\langle a \rangle)$. We claim that $a \in N(\langle ab \rangle)$. Substituting $i=1+p^{\alpha}$ in \ref{id_ab}, we get
    \begin{equation}\label{p_1a_eq2}
        \begin{split}
            (ab)^i &= [b,a]^{(1+p^{\alpha})(p^{\alpha}/2)}a^{1+p^{\alpha}}b^{1+p^{\alpha}}\\
            &= a^{1+p^{\alpha}}b^{1+p^{\alpha}} \text{ (since } (1+p^{\alpha})p^{\alpha}/2 \equiv 0  \text{ mod } p \text{)}\\ 
            &= a^{1+p^{\alpha}}b
        \end{split}
    \end{equation}
    Equating Equations \ref{p_1a_eq1} and \ref{p_1a_eq2}, we have $$a(ab)a^{-1} = a^{1+p^{\alpha}}b = (ab)^{1+p^{\alpha}} \in \langle ab \rangle.$$Therefore, we get $x$ and $y$ such that they generate $P$ and are adjacent in SNorm($P$).
\end{proof}

Theorem \ref{thm_2} follows from Lemmas \ref{com_eq_snorm_lemma}, \ref{ECG_thm}, \ref{ecg_case_1b_1c_2b} and \ref{ECG_snorm} as shown below.

\begin{thm}\label{Equality_com} 
    Let $G$ be a finite group. The commuting graph of $G$ is equal to the symmetric normaliser graph of $G$ if and only if $G$ does not have a subgroup isomorphic to SNNC groups.
\end{thm} 
\begin{proof} 
    Let $G$ be a finite group such that the two graphs of $G$ are unequal. By Lemma \ref{com_eq_snorm_lemma}, the group $G$ has a subgroup $H = \langle h, k \rangle$ which is a 2-generated $p$-group such that $|H'| =p$ and $H' = Z(H)$. Further $h, k$ are adjacent in SNorm$H$. By Theorem D of \cite{BCMV2025}, the subgroup $H$ is an exponent-critical group of type $\mathcal{B}$. By Lemma \ref{ECG_thm}, the subgroup $H$ is isomorphic to one of the groups whose parameters are mentioned in the list. Since the subgroup $H$ has elements $h$ and $k$ which generate it and are adjacent in its symmetric normaliser graph of $H$, so by Lemma \ref{ecg_case_1b_1c_2b} and Lemma \ref{ECG_snorm}, the subgroup $H$ must be an SNNC-group. \\
     Conversely, if a finite group $G$ has a subgroup isomorphic to an SNNC-group, then $G$ will have elements that satisfy \ref{snorm_property}, which means that the commuting graph of $G$ is not equal to the symmetric graph of $G$. 
\end{proof}

\section{Hierarchies and Equalities II}
Before proving Theorem \ref{thm_3}, we remark that the symmetric normaliser graph is complete if and only if the group is a Dedekind group. This is true because every subgroup of $G$ is normal in $G$ if and only if $\langle x \rangle$ is normal in $G$ for each $x \in G$. This also follows from Theorem 4.1 in \cite{DGM2025} for a directed normalizing graph of a group. The vertex set of this graph are the elements of the group, and there is a directed edge from a vertex $x$ to a vertex $y$ if the subgroup $\langle x \rangle$ is normal in the subgroup $\langle x,y \rangle$. If $x$ and $y$ are adjacent in SNorm($G$), then there is a directed edge from $x$ to $y$ and $y$ to $x$ in the directed normalizing graph of $G$.

\begin{thm}
    The symmetric normaliser graph and the nilpotent graph of a finite group $G$ are equal if and only if for an odd prime $p$, the Sylow $p$-subgroups are abelian, and for $p=2$ the Sylow $p$-subgroups are either abelian or a direct product of $Q_8$ and an elementary abelian $2$-group.
\end{thm}
\begin{proof}
    Let $G$ be a finite group such that the symmetric normaliser graph and the nilpotent graph of $G$ are equal. Let $P$ denote a Sylow $p$-subgroup of $G$. Since the nilpotent graph of $P$ is complete, the symmetric normaliser graph of $P$ is also complete, which makes subgroup $P$ a Dedekind group. By Theorem 5.3.7 of \cite{DJSR1996}, a group is Dedekind if and only if it is either abelian or the direct product of a quaternion group of order 8, an elementary abelian 2-group and an abelian group with all its elements of odd order. This gives us that for odd primes $p$, the Sylow $p$-subgroups are abelian, and the Sylow $2$-subgroups of $G$ are either abelian or a direct product of $Q_8$ and an elementary abelian $2$-group.
    \par
    Conversely, for the finite group $G$ and for an odd prime $p$, let the Sylow $p$-subgroups be abelian and for $p=2$ let the Sylow $2$-subgroups either be abelian or a direct product of $Q_8$ and an elementary abelian $2$-group. Let $x$ and $y$ be adjacent in Nilp($G$). Since $H=\langle x,y \rangle$ is a subgroup of $G$, for an odd prime $p$, the Sylow $p$-subgroups of $H$ are abelian, and for $p=2$, the Sylow $2$-subgroups of $H$ are either abelian or a direct product of $Q_8$ and an elementary abelian $2$-group. Since $H$ is nilpotent, it is a direct product of its Sylow subgroups. By Theorem 5.3.7 of \cite{DJSR1996}, the subgroup $H$ is a Dedekind group and its symmetric normaliser graph is complete. Thus $x$ and $y$ are adjacent in SNorm($G$).
\end{proof}
The proof of Theorem \ref{thm_3} shows that for a finite group $G$, its symmetric normaliser graph and nilpotent graph are the same if and only if the Sylow subgroups of $G$ are Dedekind groups. 

There are necessary and sufficient conditions mentioned in \cite{PJC2022} for a finite group to have equal power graph and enhanced power graph or equal enhanced power graph and commuting graph. Using these conditions, we have the following result.
\begin{thm}\label{equality_others}
    Let $G$ be a finite group. The enhanced power graph and the symmetric normaliser graph of $G$ are equal if and only if the Sylow $p$-subgroups of $G$ are cyclic.
\end{thm}
\begin{proof}
    Let $G$ be a finite group such that its enhanced power graph and symmetric normaliser graphs are the same. This implies that EPow($G$) = Com($G$) = SNorm($G$). By Proposition 3.2 of \cite{PJC2022}, Sylow $p$-subgroups of $G$ are cyclic or generalised quaternion groups. Since the quaternion group is an SNNC group, by Theorem \ref{Equality_com}, we must have Sylow $p$-subgroups of $G$ are cyclic. Conversely, if the Sylow $p$-subgroups of $G$ are cyclic then $G$ does not have a subgroup which is an SNNC group, and so, Com($G$) = SNorm($G$) and by Proposition 3.2 of \cite{PJC2022}, EPow($G$) = Com($G$). 
 \end{proof}

We present some concluding remarks related to other equalities in the hierarchy of graphs defined on groups. Let $G$ be a finite group.
\begin{enumerate}[(a)]
    \item From Proposition 3.2 of \cite{PJC2022} and Theorem \ref{equality_others} we get that Pow($G$) = SNorm($G$) if and only if the Sylow $p$-subgroups of $G$ are cyclic and elements of $G$ whose orders are coprime do not commute. Such a group $G$ can only have elements of prime power order, namely $G$ is an EPPO group which were first studied by G. Higman in \cite{GH1957}. To restate, we have Pow($G$) = SNorm($G$) if and only if $G$ is an EPPO group with cyclic Sylow subgroups. So $G$ is solvable and has at most two distinct primes dividing its order (see Lemma 6.3, \cite{PJCNVM2022}). The paper \cite{PJCNVM2022} also lists all EPPO groups in Theorem 1.7. From this we get that Pow($G$) = SNorm($G$) if and only if $G$ is a cyclic $p$-group or $|G|=p^{\alpha}q^{\beta}$ where $p, q$ are distinct primes and $G$ is a Frobenius or $2$-Frobenius group with cyclic Sylow subgroups.
    \item The deep commuting graph (DCom($G$)) of a finite group G was introduced in  \cite{PJCBK2023}. Two elements of $G$ are adjacent in DCom($G$) if and only if their preimages in every central extension of $G$ commute. Since $E$(DCom($G$)) $\subseteq$ $E$(Com($G$)), from Theorem \ref{thm_2} and Proposition 3.3 (b) of \cite{PJC2022} we can get a necessary and sufficient condition for DCom($G$) and SNorm($G$) to be equal.  
    \item Let $G$ not be $2$-generated then NGen($G$) is complete. So SNorm($G$) = NGen($G$) if and only if SNorm($G$) is complete which is if and only if $G$ is a Dedekind group.
    \end{enumerate}
We end with some open questions.\\ \par
{\bf Question 1} What are the finite groups $G$ of order $p^{\alpha}q^{\beta}$ for distinct primes $p,q$ that are Frobenius or $2$-Frobenius and have cyclic Sylow subgroups? If such a $G$ is Frobenius it will be a semidirect product of a Frobenius kernel $K$ isomorphic to ${\mathbb{Z}_{p^{\alpha}}}$ and a Frobenius complement $H$ isomorphic to ${\mathbb{Z}_{q^{\beta}}}$. 
\\ \par
{\bf Question 2}. Let $G$ be a non-abelian simple group. When would SNorm($G$)= NGen($G$)? or is it always the case that SNorm($G$) $\subsetneq$ NGen(G)? Proposition 3.1(c) of \cite{PJC2022} shows that for any group $H$ we have NGen($H$) = Com($H$) if and only if $H$ is abelian and not $2$-generated or a minimal non-abelian group. Since a minimal non-abelian group is always solvable, we see that for a finite non-abelian simple group $G$, we always have $E$(Com($G$)) $\subsetneq$ $E$(NGen($G$)). 
Also any SNNC-group $P$ for an odd prime $p$ is always contained in an $A_n$ for all sufficiently large $n$. So we certainly get that $E$(Com($G$)) $\subsetneq$ $E$(SNorm($G$)) when $G = A_n$ for all $n$ suitably large.\\ \par 
{\bf Question 3}. What are the necessary and sufficient conditions for SNorm$G$ to be equal to Engel($G$)?  If $G$ is a Dedekind group, then SNorm($G$) = Nilp($G$) = Engel($G$), as all are complete graphs. On the other hand, we know that if the Sylow subgroups of $G$ are Dedekind then SNorm($G$) = Nilp($G$). \\ \par
{\bf Question 4} Let $G$ be a non-Dedekind finite group whose Sylow subgroups are Dedekind. Is $E$(Nilp($G$)) $\subsetneq$ $E$(Engel($G$))? Note that for $G = S_3$, we have $E$(Pow($G$)) = $E$(EPow($G$)) = $E$(Com($G$)) = $E$(SNorm($G$)) = $E$(Nilp($G$)) $\subsetneq$ $E$(Engel($G$)).

\end{document}